\documentclass[a4paper,12pt]{article}
\usepackage{amsmath,amsthm,amssymb}
\usepackage[dvipdfmx]{graphicx}

\numberwithin{equation}{section}

\newtheorem{theorem}{Theorem}
\newtheorem{prop}{Proposition}

\newtheorem{lemma}{Lemma}
\theoremstyle{definition}

\newcommand{\R}{\mathbb{R}}
\newcommand{\N}{\mathbb{N}}
\newcommand{\Z}{\mathbb{Z}}
\newcommand{\p}{\partial}
\newcommand{\eps}{\varepsilon}

\begin{document}

\begin{center}
\Large{\textbf{
Strong instability of standing waves \\
for nonlinear Schr\"{o}dinger equations \\
with a partial confinement}}
\end{center}

\vspace{2mm}

%\begin{center}
%Dedicated to Professor Vladimir Georgiev on his 60th birthday
%\end{center}

\vspace{3mm}

\begin{center}
{\large Masahito Ohta} 
\end{center}
\begin{center}
Department of Mathematics, 
Tokyo University of Science, \\
1-3 Kagurazaka, Shinjuku-ku, Tokyo 162-8601, Japan
\end{center}

\begin{abstract}
We study the instability of standing wave solutions 
for nonlinear Schr\"{o}dinger equations 
with a one-dimensional harmonic potential in dimension $N\ge 2$. 
We prove that 
if the nonlinearity is $L^2$-critical or supercritical in dimension $N-1$, 
then any ground states are strongly unstable by blowup. 
\end{abstract}

\section{Introduction}

In this paper, we study the instability of standing wave solutions 
$e^{i\omega t}\phi_{\omega}(x)$ 
for the nonlinear Schr\"{o}dinger equation 
with a one-dimensional harmonic potential 
\begin{align}\label{nls}
i\p_t u=-\Delta u+x_N^2 u-|u|^{p-1}u,
\quad (t,x)\in \R\times \R^N, 
\end{align}
where $N\ge 2$, 
$x_N$ is the $N$-th component of $x=(x_1,...,x_N)\in \R^N$, 
$\Delta$ is the Laplacian in $x$, 
and $1<p<1+4/(N-2)$. 
Here, $1+4/(N-2)$ stands for $\infty$ if $N=2$. 

The Cauchy problem for \eqref{nls} is locally well-posed in the energy space $X$
(see \cite[Theorem 9.2.6]{caz}). 
Here, the energy space $X$ for \eqref{nls} is defined by 
$$X=\{v\in H^1(\R^N):  x_N v \in L^2 (\R^N)\}$$
with the norm 
$$\|v\|_X=\left( \|\nabla v\|_{L^2}^2+\|v\|_{L^2}^2+\|x_Nv\|_{L^2}^2 \right)^{1/2}.$$

\begin{prop}
Let $1<p<1+4/(N-2)$. 
For any $u_0\in X$ there exist $T_{\max}=T_{\max}(u_0)\in (0,\infty]$ 
and a unique maximal solution $u\in C([0,T_{\max}), X)\cap C^1([0,T_{\max}), X^*)$ 
of \eqref{nls} with initial condition $u(0)=u_0$. 
The solution $u(t)$ is maximal in the sense that 
if $T_{\max}<\infty$, 
then $\displaystyle{\|u(t)\|_{X}\to \infty}$ as 
$t\nearrow T_{\max}$. 

Moreover, the solution $u(t)$ satisfies the conservation laws 
\begin{align}\label{conservation}
\|u(t)\|_{L^2}^2=\|u_0\|_{L^2}^2, \quad E(u(t))=E(u_0)
\end{align}
for all $t\in [0,T_{\max})$, 
where the energy $E$ is defined by 
\begin{align*}
E(v)=\frac{1}{2}\|\nabla v\|_{L^2}^2
+\frac{1}{2}\|x_N v\|_{L^2}^2
-\frac{1}{p+1}\|v\|_{L^{p+1}}^{p+1}. 
\end{align*}
\end{prop}

Next, we consider the stationary problem 
\begin{align}\label{sp}
-\Delta \phi+x_N^2 \phi+\omega \phi-|\phi|^{p-1}\phi=0, 
\quad x\in \R^N,
\end{align}
where $\omega\in \R$. 
Note that 
if $\phi(x)$ solves \eqref{sp}, 
then $e^{i\omega t} \phi(x)$ is a solution of \eqref{nls}. 
Moreover, \eqref{sp} can be written as $S_{\omega}'(\phi)=0$, 
where 
\begin{align*}
S_{\omega}(v)
&=E(v)+\frac{\omega}{2}\|v\|_{L^2}^2 \\
&=\frac{1}{2}\|\nabla v\|_{L^2}^2
+\frac{1}{2}\|x_N v\|_{L^2}^2
+\frac{\omega}{2}\|v\|_{L^2}^2
-\frac{1}{p+1}\|v\|_{L^{p+1}}^{p+1}
\end{align*}
is the action. 
The set of all ground states for \eqref{sp} is defined by 
\begin{equation} \label{GS1}
\mathcal{G}_{\omega}
=\{\phi \in \mathcal{A}_{\omega}: S_{\omega}(\phi)\le S_{\omega}(v)
\hspace{2mm} \mbox{for all} \hspace{2mm}
v\in \mathcal{A}_{\omega}\},
\end{equation}
where 
\begin{equation*} 
\mathcal{A}_{\omega}
=\{v \in X: S_{\omega}'(v)=0,~v\ne 0\}
\end{equation*}
is the set of all nontrivial solutions for \eqref{sp}. 

Then, we have the following result on the existence of ground states for \eqref{sp}. 

\begin{prop} \label{prop-GS}
Let $1<p<1+4/(N-2)$ and $\omega\in (-1,\infty)$. 
Then, the set $\mathcal{G}_{\omega}$ is not empty, and it is characterized by 
\begin{equation} \label{GS2}
\mathcal{G}_{\omega}
=\{v\in X: S_{\omega}(v)=d(\omega),~ K_{\omega}(v)=0,~v\ne 0\},
\end{equation}
where 
\begin{align*}
K_{\omega}(v)
=\p_{\lambda} S_{\omega} (\lambda v)|_{\lambda=1} 
=\|\nabla v\|_{L^2}^2
+\|x_N v\|_{L^2}^2
+\omega \|v\|_{L^2}^2
-\|v\|_{L^{p+1}}^{p+1}
\end{align*}
is the Nehari functional, and 
\begin{equation} \label{GS3}
d(\omega)=\inf\{S_{\omega}(v):
v\in X,~ K_{\omega}(v)=0,~ v\ne 0\}. 
\end{equation}
\end{prop}

Although Proposition \ref{prop-GS} can be proved 
by the standard concentration compactness argument, 
for the sake of completeness, 
we give the proof of Proposition \ref{prop-GS} in Section 3. 

Here, we remark that by Heisenberg's inequality
\begin{equation} \label{He0}
\|v\|_{L^2}^2\le 2\|\p_{N}v\|_{L^2} \|x_Nv\|_{L^2},
\end{equation}
for any $\omega\in (-1,\infty)$ 
there exist positive constants $C_1(\omega)$ and $C_2(\omega)$ such that 
\begin{equation} \label{He1}
C_1(\omega)  \|v\|_X^2
\le \|\nabla v\|_{L^2}^2+\|x_N v\|_{L^2}^2+\omega \|v\|_{L^2}^2
\le C_2(\omega) \|v\|_X^2
\end{equation}
for all $v\in X$. 

Now we state our main result in this paper. 

\begin{theorem}\label{thm1}
Assume that $N\ge 2$, $1+4/(N-1)\le p<1+4/(N-2)$, 
and let $\phi_{\omega}\in \mathcal{G}_{\omega}$ for $\omega\in (-1,\infty)$. 
Then, for any $\omega\in (-1,\infty)$, 
the standing wave solution $e^{i\omega t}\phi_{\omega}$ of \eqref{nls} 
is strongly unstable in the following sense. 
For any $\eps>0$ there exists $u_0\in X$ such that 
$\|u_0-\phi_{\omega}\|_X<\eps$ and 
the solution $u(t)$ of \eqref{nls} with $u(0)=u_0$ blows up in finite time. 
\end{theorem}

Notice that Theorem \ref{thm1} covers the physically relevant case 
$N=3$ and $p=3$ as a borderline case. 

Here, we recall some known results related to Theorem \ref{thm1}. 
First, we consider the nonlinear Schr\"{o}dinger equations without potential
\begin{align}\label{nls0}
i\p_t u=-\Delta u-|u|^{p-1}u,
\quad (t,x)\in \R\times \R^N, 
\end{align}
where $1<p<1+4/(N-2)$. 
For any $\omega\in (0,\infty)$, 
there exists a unique positive radial solution $\phi_{\omega}(x)$ of 
the stationary problem 
\begin{align*}
-\Delta \phi+\omega \phi-|\phi|^{p-1}\phi=0, 
\quad x\in \R^N
\end{align*}
(see \cite{Kwong} for the uniqueness). 
When $1<p<1+4/N$, 
the standing wave solution $e^{i\omega t}\phi_{\omega}$ of \eqref{nls0} is 
orbitally stable for all $\omega>0$ (see \cite{CL}). 
While, if $1+4/N\le p<1+4/(N-2)$, 
then 
the standing wave solution $e^{i\omega t}\phi_{\omega}$ of \eqref{nls0} is 
strongly unstable for all $\omega>0$ 
(see \cite{BC} and also \cite[Theorem 8.2.2]{caz}). 

Next, we consider the nonlinear Schr\"{o}dinger equations 
with a harmonic potential
\begin{align}\label{nls1}
i\p_t u=-\Delta u+|x|^2u-|u|^{p-1}u,
\quad (t,x)\in \R\times \R^N, 
\end{align}
where $1<p<1+4/(N-2)$. 
For any $\omega\in (-N,\infty)$, 
there exists a unique positive radial solution $\phi_{\omega}(x)$ of 
the stationary problem 
\begin{align*}
-\Delta \phi+|x|^2 \phi+\omega \phi-|\phi|^{p-1}\phi=0, 
\quad x\in \R^N
\end{align*}
(see \cite{HO1,HO2} for the uniqueness). 

When $\omega$ is sufficiently close to $-N$, 
the standing wave solution $e^{i\omega t} \phi_{\omega}$ of \eqref{nls1} is orbitally stable 
for any $1<p<1+4/(N-2)$ (see \cite{FO1}). 
We remark that $N$ is the first eigenvalue of $-\Delta+|x|^2$. 

On the other hand, when $\omega$ is sufficiently large, 
the standing wave solution $e^{i\omega t} \phi_{\omega}$ of \eqref{nls1} is 
orbitally stable for the case $1<p\le 1+4/N$ (see \cite{fuk05,FO1}), 
and it is strongly unstable for the case $1+4/N<p<1+4/(N-2)$ 
(see \cite{oht16} and also \cite{FO2} for an earlier result on the orbital instability). 

Finally, we consider the nonlinear Schr\"{o}dinger equations 
with a partial confinement of the form 
\begin{align}\label{nls2}
i\p_t u=-\Delta u+(x_1^2+\cdots+x_d^2) u-|u|^{p-1}u,
\quad (t,x)\in \R\times \R^N, 
\end{align}
where $N\ge 2$, $1\le d\le N-1$, 
$x=(x_1,...,x_d,x_{d+1},...,x_N)$. 
The typical case is that $N=3$ and $d=2$. 
Recently, 
Bellazzini, Boussa\"id, Jeanjean and Visciglia \cite{BBJV} 
constructed orbitally stable standing wave solutions of \eqref{nls2} 
for the case 
\begin{equation} \label{UBp}
1+4/N<p<\min\{1+4/(N-d),~ 1+4/(N-2)\}
\end{equation}
(see Theorem 1 and Remark 1.9 of \cite{BBJV}). 
It should be remarked that 
the bottom of the spectrum 
of $-\Delta+(x_1^2+\cdots+x_d^2)$ is not an eigenvalue, 
so that unlike \eqref{nls1} with a complete confinement, 
the existence of stable standing wave solutions for \eqref{nls2} 
is highly nontrivial in the $L^2$-supercritical case $p>1+4/N$. 

We also remark that for the case $d\ge 2$, 
the assumption \eqref{UBp} becomes 
$1+4/N<p<1+4/(N-2)$. 
On the other hand, for the case $d=1$, 
the assumption \eqref{UBp} becomes 
$1+4/N<p<1+4/(N-1)$, 
and there is a chance to consider the case 
$1+4/(N-1)\le p<1+4/(N-2)$. 
This is our main motivation for Theorem \ref{thm1} in the present paper
(see also \cite{ACS, CG, TTV} for related results). 

Although it is not clear whether the standing wave solutions 
constructed by \cite{BBJV} are ground states 
in the sense of \eqref{GS1}
(see Definition 1.1 and Remark 1.10 of \cite{BBJV}), 
it would be safe to conclude from  our Theorem \ref{thm1}  that 
the upper bound on $p$ in \eqref{UBp} is optimal 
for the existence of stable standing wave solutions of \eqref{nls2}. 

The rest of the paper is organized as follows. 
In Section 2, we give the proof of Theorem \ref{thm1}. 
The proof is based on a virial type identity \eqref{virial} 
associated with the scaling \eqref{scale}, 
the characterization of ground states \eqref{GS2} 
by the minimization problem on the Nehari manifold, 
and Lemma \ref{lem1} below. 
We remark that 
the classical method by Berestycki and Cazenave \cite{BC} 
is not applicable to \eqref{nls} directly. 
Instead, we use and modify the ideas of Zhang \cite{zhang02} and Le Coz \cite{lec}, 
which give an alternative approach to the strong instability 
(see also \cite{oht16, OY1, OY2} for recent developments). 

In Section 3, we give the proof of Proposition \ref{prop-GS}. 
The proof is based on the standard concentration compactness argument. 

\section{Proof of Theorem \ref{thm1}}

We define 
\begin{align*}
\Sigma=\{v\in H^1(\R^N): |x|v\in L^2(\R^N)\}. 
\end{align*}

First, we derive a virial type identity. 

\begin{prop} \label{prop-V}
Let $1<p<1+4/(N-2)$. 
If $u_0\in \Sigma$, 
then the solution $u(t)$ of \eqref{nls} with $u(0)=u_0$ satisfies 
$u\in C([0,T_{\max}),\Sigma)$. 
Moreover, the function 
\begin{align*}
t\mapsto 
F(t)=\sum_{j=1}^{N-1} \int_{\R^N} x_j^2 |u(t,x)|^2\,dx
\end{align*}
is in $C^2[0,T_{\max})$, 
and  satisfies 
\begin{equation} \label{virial}
F''(t)=16 P(u(t))
\end{equation}
for all $t\in [0,T_{\max})$,
where 
\begin{align*}
P(v)
=\frac{1}{2} \sum_{j=1}^{N-1} \|\p_j v\|_{L^2}^2
-\frac{\alpha}{2(p+1)}\|v\|_{L^{p+1}}^{p+1}, \quad 
\alpha=\dfrac{(N-1)(p-1)}{2}. 
\end{align*}
\end{prop}

\begin{proof}
We state formal calculations for the identity \eqref{virial} only. 
These formal calculations can be justified by the classical regularization argument 
as in \cite[Proposition 6.5.1]{caz} (see also \cite{mar}). 

Let $u(t,x)$ be a smooth solution of \eqref{nls}. 
Then, we have 
\begin{align*}
&F'(t)
=2 \sum_{j=1}^{N-1} \, {\rm Im} \int_{\R^N} 
x_j^2 \overline{u} 
\left(-\Delta u+x_N^2 u-|u|^{p-1}u \right)\,dx \\
&=-2 \sum_{j=1}^{N-1} \, {\rm Im} \int_{\R^N} 
x_j^2 \overline{u} \Delta u \,dx 
=4 \sum_{j=1}^{N-1} \, {\rm Im} \int_{\R^N} \overline{u} x_j \p_j u\,dx. 
\end{align*}

Moreover, we have 
\begin{align*}
F''(t)
=-4 \sum_{j=1}^{N-1} \, {\rm Im} \int_{\R^N} 
\p_t u 
\left( 2x_j \p_j \overline{u}+\overline{u} \right)\,dx. 
\end{align*}

Here, we consider the scaling 
\begin{equation} \label{scale}
v^{\lambda}(x)=\lambda^{(N-1)/2} v(\lambda x_1,...,\lambda x_{N-1},x_N)
\end{equation}
for $\lambda>0$ and $x=(x_1,...,x_{N-1},x_N)\in \R^N$.  
Then, we have 
\begin{align*}
&\p_{\lambda} v^{\lambda}(x)|_{\lambda=1}
=\sum_{j=1}^{N-1} x_j \p_j v(x)
+\frac{N-1}{2}v(x), \\
&E(v^{\lambda})
=\frac{\lambda^2}{2} \sum_{j=1}^{N-1} \|\p_j v\|_{L^2}^2
-\frac{\lambda^{\alpha}}{p+1} \|v\|_{L^{p+1}}^{p+1}
+\frac{1}{2} \|\p_N v\|_{L^2}^2+\frac{1}{2}\|x_N v\|_{L^2}^2, 
\end{align*}
and 
\begin{align*}
P(v)=\dfrac{1}{2}\, \p_{\lambda} E(v^{\lambda}) |_{\lambda=1}. 
\end{align*}

Thus, we have 
\begin{align*}
F''(t)
&=8 \,{\rm Re} \int_{\R^N} 
\left( -\Delta u+x_N^2u-|u|^{p-1}u \right)
\overline{\p_{\lambda} u^{\lambda} |_{\lambda=1}} \,dx \\
&=8 \p_{\lambda} E(u^{\lambda}) |_{\lambda=1}
=16 P(u(t)). 
\end{align*}

As stated above, 
these formal calculations can be justified by the regularization argument. 
\end{proof}

Notice that 
\begin{align*}
\alpha=\dfrac{(N-1)(p-1)}{2} \ge 2
\end{align*}
for the case $1+4/(N-1)\le p<1+4/(N-2)$. 

The following lemma is a modification of the ideas of Zhang \cite{zhang02} and Le Coz \cite{lec} 
(see also \cite{oht16, OY1, OY2}). 

\begin{lemma} \label{lem1}
Assume that $1+4/(N-1)\le p<1+4/(N-2)$ and $\omega\in (-1,\infty)$. 
If $v\in X$ satisfies $P(v)\le 0$ and $v\ne 0$, 
then $d(\omega)\le S_{\omega}(v)-P(v)$. 
\end{lemma}

\begin{proof}
Since $\omega>-1$ and $v\ne 0$, 
by Heisenberg's inequality \eqref{He0}, 
we have 
\begin{align*}
C_0:=\|\p_N v\|_{L^2}^2 +\|x_N v\|_{L^2}^2+\omega \|v\|_{L^2}^2 
\ge (\omega+1)\|v\|_{L^2}^2>0. 
\end{align*}

Then, it follows from $P(v)\le 0$ that 
\begin{align*}
K_{\omega}(v^{\lambda})
&=\lambda^2 \sum_{j=1}^{N-1} \|\p_j v\|_{L^2}^2
-\lambda^{\alpha} \|v\|_{L^{p+1}}^{p+1}+C_0 \\
&\le \left(\frac{\alpha \lambda^2}{p+1}-\lambda^{\alpha}\right)
 \|v\|_{L^{p+1}}^{p+1}+C_0
\end{align*}
for $\lambda>0$. 
Since $\alpha\ge 2$ and $v\ne 0$, 
there exists $\lambda_0\in (0,\infty)$ such that 
$K_{\omega}(v^{\lambda_0})=0$. 
Here, we remark that 
\begin{align*}
\frac{\alpha \lambda^2}{p+1}-\lambda^{\alpha}
=-\frac{p-1}{p+1} \lambda^2  
\end{align*}
for the case $\alpha=2$. 

Then, 
by the definition \eqref{GS3} of $d(\omega)$, 
we have $d(\omega)\le S_{\omega} (v^{\lambda_0})$. 

Moreover, since $\alpha\ge 2$, the function 
\begin{align*}
(0,\infty)\ni \lambda\mapsto 
S_{\omega}(v^{\lambda})-\lambda^2 P(v)
=\frac{\alpha \lambda^2-2 \lambda^{\alpha}}{2(p+1)} \|v\|_{L^{p+1}}^{p+1}
+\frac{C_0}{2}
\end{align*}
attains its maximum at $\lambda=1$. 

Thus,  since $P(v)\le 0$ again, we have 
\begin{align*}
d(\omega)\le S_{\omega} (v^{\lambda_0})
\le S_{\omega} (v^{\lambda_0})-\lambda_0^2 P(v)
\le S_{\omega}(v)-P(v).
\end{align*}

This completes the proof. 
\end{proof}

Once we have obtained Lemma \ref{lem1}, 
the rest of the proof is the same as in the classical argument of 
Berestycki and Cazenave \cite{BC}. 

\begin{lemma}\label{lem2}
Assume that $1+4/(N-1)\le p<1+4/(N-2)$ and $\omega\in (-1,\infty)$. 
The set 
\begin{align*}
\mathcal{B}_{\omega}
=\{ v\in X: S_{\omega}(v)<d(\omega),~ P(v)<0 \}
\end{align*} 
is invariant under the flow of \eqref{nls}. 
That is, if $u_0 \in \mathcal{B}_{\omega}$, 
then the solution $u(t)$ of \eqref{nls} with $u(0)=u_0$ satisfies 
$u(t)\in \mathcal{B}_{\omega}$ for all $t\in [0,T_{\max})$. 
\end{lemma}

\begin{proof}
This follows from the conservation laws \eqref{conservation}, Lemma \ref{lem1}, 
and the continuity of the function $t\mapsto P(u(t))$. 
\end{proof}

\begin{theorem}\label{thm2}
Assume that $1+4/(N-1)\le p<1+4/(N-2)$ and $\omega\in (-1,\infty)$. 
If $u_0\in \mathcal{B}_{\omega}\cap \Sigma$, 
then the solution $u(t)$ of \eqref{nls} with $u(0)=u_0$ blows up in finite time. 
\end{theorem}

\begin{proof}
Let $u_0\in \mathcal{B}_{\omega}\cap \Sigma$ 
and let $u(t)$ be the solution of \eqref{nls} with $u(0)=u_0$. 
Then, it follows from Lemma \ref{lem2} and Proposition \ref{prop-V} that 
$u(t)\in \mathcal{B}_{\omega}\cap \Sigma$ for all $t\in [0,T_{\max})$. 

Moreover, by the virial identity \eqref{virial}, 
the conservation laws \eqref{conservation}
and Lemma \ref{lem1}, we have 
\begin{align*}
&\frac{1}{16}\frac{d^2}{dt^2} 
\sum_{j=1}^{N-1} \int_{\R^N} x_j^2 |u(t,x)|^2\,dx
=P(u(t)) \\
&\le S_{\omega} (u(t))-d(\omega) 
=S_{\omega}(u_0)-d(\omega)<0
\end{align*}
for all $t\in [0,T_{\max})$. 
This implies $T_{\max}<\infty$. 
\end{proof}

Finally, we give the proof of Theorem \ref{thm1}. 

\vspace{2mm} \noindent 
{\bf Proof of Theorem \ref{thm1}}. 
First, 
by the elliptic regularity theory, 
we see that $\phi_{\omega}\in \Sigma$
(see, e.g., \cite[Theorem 8.1.1]{caz}). 

Next, since 
$S_{\omega}'(\phi_{\omega})=0$, 
the function 
\begin{align*}
&(0,\infty)\ni \lambda\mapsto \\ 
&S_{\omega}(\lambda \phi_{\omega})
=\lambda^2 \left\{\frac{1}{2} \|\nabla \phi_{\omega}\|_{L^2}^2
+\frac{1}{2}\|x_N \phi_{\omega}\|_{L^2}^2
+\frac{\omega}{2} \|\phi_{\omega}\|_{L^2}^2 \right\}
-\frac{\lambda^{p+1}}{p+1} \|\phi_{\omega}\|_{L^{p+1}}^{p+1}
\end{align*}
attains its maximum at $\lambda=1$. 
Thus, we have 
\begin{align*}
S_{\omega}(\lambda \phi_{\omega})
<S_{\omega}(\phi_{\omega})=d(\omega)
\end{align*}
for all $\lambda>1$. 
Moreover, 
since $P(\phi_{\omega})=0$, we have 
\begin{align*} 
P(\lambda \phi_{\omega})
=\frac{\lambda^2}{2} \sum_{j=1}^{N-1} \|\p_j \phi_{\omega}\|_{L^2}^2
-\frac{\alpha \lambda^{p+1}}{2(p+1)}\|\phi_{\omega}\|_{L^{p+1}}^{p+1}<0
\end{align*}
for all $\lambda>1$. 

Therefore, we see that 
$\lambda \phi_{\omega}\in \mathcal{B}_{\omega}\cap \Sigma$
for  all $\lambda>1$, 
and it follows from Theorem \ref{thm2} that 
the solution $u(t)$ of \eqref{nls} with $u(0)=\lambda \phi_{\omega}$ 
blows up in finite time. 
Hence, the result follows, 
since $\lambda \phi_{\omega}\to \phi_{\omega}$ in $X$ as $\lambda\to 1$. 
\qed \vspace{2mm} 

\section{Proof of Proposition \ref{prop-GS}}

In this section, we prove Proposition \ref{prop-GS} 
by using the standard concentration compactness argument. 
Throughout this section, we assume that $1<p<1+4/(N-2)$ and $\omega\in (-1,\infty)$. 

We define 
\begin{align}
J_{\omega}(v)
&=S_{\omega}(v)-\frac{1}{p+1} K_{\omega} (v) \label{J1} \\
&=\frac{p-1}{2(p+1)} \left( \|\nabla v\|_{L^2}^2+\|x_N v\|_{L^2}^2+\omega \|v\|_{L^2}^2 \right). 
\nonumber 
\end{align}

Note that by \eqref{He1}, 
there exists a positive constant $C_0$ 
depending only on $\omega$ and $p$ such that 
\begin{equation} \label{He2}
J_{\omega}(v)\ge C_0 \|v\|_{X}^2, \quad 
v\in X. 
\end{equation}

We also remark that 
by \eqref{J1} and \eqref{GS3}, 
we have 
\begin{equation} \label{J2}
d(\omega)
=\inf \{J_{\omega}(v): v\in X,~ K_{\omega}(v)=0,~ v\ne 0\}. 
\end{equation}

\begin{lemma} \label{lem3}
$d(\omega)>0$. 
\end{lemma}

\begin{proof}
Let $v\in X$ satisfy $K_{\omega}(v)=0$ and $v\ne 0$. 

Then, by $K_{\omega}(v)=0$, the Sobolev inequality 
and \eqref{He2}, 
there exist positive constants $C_1$ and $C_2$ 
depending only on $N$, $p$ and $\omega$ such that 
\begin{align*}
J_{\omega}(v)
=\frac{p-1}{2(p+1)} \| v\|_{L^{p+1}}^{p+1}
\le C_1 \|v\|_{H^1}^{p+1}
\le C_1 \|v\|_{X}^{p+1}
\le C_2 J_{\omega}(v)^{(p+1)/2}. 
\end{align*}

Since $v\ne 0$, we have $J_{\omega}(v)>0$ 
and $J_{\omega}(v)^{(p-1)/2}\ge 1/C_2$. 

Thus, by \eqref{J2}, we have 
\begin{align*}
d(\omega)\ge \frac{1}{C_2^{2/(p-1)}} >0. 
\end{align*}

This completes the proof. 
\end{proof}

\begin{lemma} \label{lem4}
If $v\in X$ satisfies $K_{\omega}(v)<0$, 
then $d(\omega)<J_{\omega}(v)$. 
\end{lemma}

\begin{proof}
Since $K_{\omega}(v)<0$ and 
\begin{equation} \label{K1}
K_{\omega} (\lambda v)
=\lambda^2 
\left( \|\nabla v\|_{L^2}^2+\|x_N v\|_{L^2}^2+\omega \|v\|_{L^2}^2 \right)
-\lambda^{p+1} \|v\|_{L^{p+1}}^{p+1}
\end{equation}
for $\lambda>0$, 
there exists $\lambda_0\in (0,1)$ such that $K_{\omega}(\lambda_0 v)=0$. 

Thus, by \eqref{J2} and \eqref{J1}, we have 
\begin{align*}
d(\omega)
\le J_{\omega} (\lambda_0 v)
=\lambda_0^2 J_{\omega}(v)
<J_{\omega}(v). 
\end{align*}

This completes the proof. 
\end{proof}

The following lemma is a variant of the classical result of Lieb \cite{Lieb}
(see also \cite[Lemma 3.4]{BBJV}). 

\begin{lemma} \label{lem5} 
Assume that a sequence $(u_n)_{n\in \N}$ is bounded in $X$, 
and satisfies 
\begin{align*}
\limsup_{n\to \infty} \|u_n\|_{L^{p+1}}^{p+1}>0. 
\end{align*}
Then, there exist a sequence $(y^n)_{n\in \N}$ in $\R^{N-1}$ and $u\in X\setminus \{0\}$ such that 
$(\tau_{y^n} u_n)_{n\in \N}$ has a subsequence which converges to $u$ weakly in $X$. 

Here we define 
\begin{align*}
\tau_y v(x)=v(x_1-y_1,...,x_{N-1}-y_{N-1},x_N)
\end{align*}
for $x=(x_1,...,x_{N-1},x_N)\in \R^N$ and $y=(y_1,...,y_{N-1})\in \R^{N-1}$. 
\end{lemma}

\begin{proof}
Without loss of generality, we may assume that 
\begin{align*}
C_1:=\inf_{n\in \N} \|u_n\|_{L^{p+1}}^{p+1}>0. 
\end{align*}

Moreover, we put 
\begin{align*}
C_2:=\sup_{n\in \N} \|u_n\|_X^2, \quad 
C_3:=\frac{C_2+1}{C_1}, 
\end{align*}
and for $y=(y_1,...,y_{N-1})\in \Z^{N-1}$, we define 
\begin{align*}
Q_y
&=(y_1,y_1+1)\times \cdots \times (y_{N-1},y_{N-1}+1)\times \R \\
&=\{(x_1,...,x_{N-1},x_N)\in \R^N: y_j<x_j<y_j+1 \hspace{1mm} (j=1,...,N-1)\}. 
\end{align*}

Then, by the definition of $C_3$, we see that 
for any $n\in \N$, there exists $y^n\in \Z^{N-1}$ such that 
\begin{align*}
\|u_n\|_{X(Q_{y^n})}^2
<C_3 \|u_n\|_{L^{p+1}(Q_{y^n})}^{p+1}. 
\end{align*}
where we put 
\begin{align*}
\|v\|_{X(Q_y)}^2
=\|v\|_{H^1(Q_y)}^2+\|x_N v\|_{L^2(Q_y)}^2.
\end{align*}

Here, we define $v_n=\tau_{-y^n} u_n$. 
Then, we have 
\begin{align*}
\|v_n\|_{X(Q_{0})}^2<C_3 \|v_n\|_{L^{p+1}(Q_{0})}^{p+1}
\end{align*}
for all $n\in \N$. 
In particular,  
$\|v_n\|_{L^{p+1}(Q_{0})}^{p+1}>0$ for all $n\in \N$. 

Moreover, by the Sobolev inequality, we have 
\begin{align*}
C_4 \|v_n\|_{L^{p+1}(Q_{0})}^2
\le \|v_n\|_{H^1(Q_{0})}^2
\le \|v_n\|_{X(Q_{0})}^2
\end{align*}
for all $n\in \N$, 
where $C_4$ is a positive constant depending only on $N$ and $p$. 

Thus, we have 
\begin{equation} \label{lieb}
\frac{C_4}{C_3}<\|v_n\|_{L^{p+1}(Q_{0})}^{p-1}, \quad n\in \N. 
\end{equation}

Since $(v_n)_{n\in \N}$ is bounded in $X$, 
there exist a subsequence $(v_{n'})$ of $(v_n)$ and $u\in X$ such that 
$(v_{n'})$ converges to $u$ weakly in $X$. 

Finally, since the embedding $X(Q_0) \hookrightarrow L^{p+1}(Q_0)$ is compact, 
it follows from \eqref{lieb} that 
\begin{equation*}
0<\frac{C_4}{C_3}\le 
\|u\|_{L^{p+1}(Q_{0})}^{p-1},
\end{equation*}
which implies $u\ne 0$. 
This completes the proof. 
\end{proof}

We define the set of all minimizers for \eqref{GS3} by 
\begin{equation*}
\mathcal{M}_{\omega}
=\{v\in X: S_{\omega}(v)=d(\omega),~ K_{\omega}(v)=0,~v\ne 0\}. 
\end{equation*}

\begin{lemma}
The set $\mathcal{M}_{\omega}$ is not empty. 
\end{lemma}

\begin{proof}
Let $(u_n)$ be a sequence in $X$ such that 
$K_{\omega}(u_n)=0$, $u_n\ne 0$ for all $n\in \N$, 
and $S_{\omega}(u_n)\to d(\omega)$. 

Then, by \eqref{He2} and $J_{\omega}(u_n)=S_{\omega}(u_n)\to d(\omega)$, 
we see that the sequence $(u_n)_{n\in \N}$ is bounded in $X$. 

Moreover, it follows from $K_{\omega}(u_n)=0$ and 
Lemma \ref{lem3} that 
\begin{align*}
\|u_n\|_{L^{p+1}}^{p+1}
=\frac{2(p+1)}{p-1} J_{\omega} (u_n)
\to \frac{2(p+1)}{p-1} d(\omega)>0. 
\end{align*}

Thus, by Lemma \ref{lem5}, 
there exist a sequence $(y^n)$ in $\R^{N-1}$, 
a subsequence of $(\tau_{y^n} u_n)$, 
which is denoted by $(v_n)$, 
and $v\in X\setminus \{0\}$ such that 
$(v_n)$ converges to $v$ weakly in $X$. 
By the weakly lower semicontinuity of $J_{\omega}$, 
we have 
\begin{equation} \label{Jd}
J_{\omega}(v)
\le \liminf_{n\to \infty} J_{\omega}(v_n)=d(\omega). 
\end{equation}

Moreover, by the Brezis-Lieb Lemma (see \cite{BL}), 
we have 
\begin{equation*}
K_{\omega}(v_n)-K_{\omega}(v_n-v)\to K_{\omega}(v), 
\end{equation*}
which implies $K_{\omega}(v)\le 0$. 

Indeed, suppose that $K_{\omega}(v)>0$. 
Since $K_{\omega}(v_n)=0$, 
we have $K_{\omega}(v_n-v)<0$ for large $n$. 
Then, by Lemma \ref{lem4}, 
we have $d(\omega)<J_{\omega}(v_n-v)$, and 
\begin{equation*}
J_{\omega}(v)
=\lim_{n\to \infty}
\{ J_{\omega}(v_n)-J_{\omega}(v_n-v) \} \le 0.
\end{equation*}
On the other hand, 
by $v\ne 0$ and \eqref{He2}, 
we have $J_{\omega}(v)>0$. 
This is a contradiction. 
Thus, we obtain $K_{\omega}(v)\le 0$. 

Furthermore, by Lemma \ref{lem4} and \eqref{Jd}, 
we have $K_{\omega}(v)=0$. 
Since $v\ne 0$ again, 
it follows from \eqref{GS3} and \eqref{Jd} that 
\begin{align*}
d(\omega)\le S_{\omega}(v)=J_{\omega}(v)\le d(\omega). 
\end{align*}

Hence, we have $S_{\omega}(v)=d(\omega)$ 
and $v\in \mathcal{M}_{\omega}$. 

This completes the proof. 
\end{proof}

\begin{lemma} \label{lem7} 
$\mathcal{M}_{\omega}\subset \mathcal{G}_{\omega}$. 
\end{lemma}

\begin{proof}
Let $\phi \in \mathcal{M}_{\omega}$. 
Then, there exists a Lagrange multiplier $\mu \in \R$ such that 
$S_{\omega}'(\phi)=\mu K_{\omega}'(\phi)$. 
Thus, we have 
\begin{align*}
0=K_{\omega}(\phi)=\langle S_{\omega}'(\phi), \phi \rangle
=\mu \langle K_{\omega}'(\phi), \phi \rangle. 
\end{align*}

Here, by \eqref{K1}, $K_{\omega}(\phi)=0$ and $\phi \ne 0$, we have 
\begin{align*}
\langle K_{\omega}'(\phi), \phi \rangle
&=\p_{\lambda} K_{\omega}(\lambda \phi)|_{\lambda=1} \\
&=2 \left( \|\nabla \phi \|_{L^2}^2+\|x_N \phi \|_{L^2}^2+\omega \| \phi \|_{L^2}^2 \right)
-(p+1) \| \phi \|_{L^{p+1}}^{p+1} \\
&=-(p-1) \| \phi \|_{L^{p+1}}^{p+1}<0. 
\end{align*}

Thus, we have $\mu=0$ and $S_{\omega}'(\phi)=0$, 
which shows that $\phi \in \mathcal{A}_{\omega}$. 

Moreover, for any $v\in \mathcal{A}_{\omega}$, 
we have $K_{\omega}(v)=\langle S_{\omega}'(v), v \rangle=0$ and $v\ne 0$, 
so it follows from the definition \eqref{GS3} of $d(\omega)$ that 
$S_{\omega}(\phi)=d(\omega)\le S_{\omega} (v)$. 

Therefore, we have $\phi \in \mathcal{G}_{\omega}$, 
and we conclude that $\mathcal{M}_{\omega}\subset \mathcal{G}_{\omega}$. 
\end{proof}

Finally, we give the proof of Proposition \ref{prop-GS}. 

\vspace{2mm} \noindent 
{\bf Proof of Proposition \ref{prop-GS}}. 
By Lemma \ref{lem7}, 
it is enough to show that 
$\mathcal{G}_{\omega}\subset \mathcal{M}_{\omega}$. 

Let $\phi \in \mathcal{G}_{\omega}$. 
By Lemma \ref{lem7}, we can take an element $v\in \mathcal{M}_{\omega}$. 
Then, since $v\in \mathcal{A}_{\omega}$ and $\phi \in \mathcal{G}_{\omega}$, 
by the definition \eqref{GS1} of $\mathcal{G}_{\omega}$, 
we have 
\begin{align*}
S_{\omega}(\phi)\le S_{\omega}(v)=d(\omega).
\end{align*}

On the other hand, since $\phi$ satisfies $K_{\omega}(\phi)=0$ and $\phi\ne 0$, 
by the definition \eqref{GS3} of $d(\omega)$, 
we have $d(\omega)\le S_{\omega}(\phi)$. 

Hence, we have $S_{\omega}(\phi)=d(\omega)$ 
and $\phi \in \mathcal{M}_{\omega}$. 

This completes the proof. 
\qed \vspace{2mm} 

\vspace{2mm} \noindent 
{\bf Acknowledgments}. 
The author thanks Louis Jeanjean and Noriyoshi Fukaya 
for useful discussions and suggestions. 
This work was supported by JSPS KAKENHI Grant Number 15K04968.

{\it E-mail address}: mohta@rs.tus.ac.jp 

\end{document}